\documentclass[10pt]{amsart}
\usepackage{amsfonts}
\usepackage{amsthm,amsmath,amssymb,mathrsfs,yhmath}
\usepackage{tikz}
\usepackage{pstricks,pst-node, pst-text,pst-3d,graphicx}
\usepackage{xcolor}
\usepackage{anysize}
\usepackage{enumerate}
\usepackage{verbatim}
\usepackage[active]{srcltx}

\marginsize{2.5cm}{3cm}{3cm}{4cm}
\usetikzlibrary{trees}
\newtheorem{theorem}{Theorem}
\newtheorem{corollary}[theorem]{Corollary}
\newtheorem{lemma}[theorem]{Lemma}
\newtheorem{proposition}[theorem]{Proposition}
\newtheorem{example}[theorem]{Example}

\theoremstyle{definition}

\newtheorem{claim}{Claim}
\theoremstyle{remark}
\newtheorem{remark}{Remark}
\newtheorem*{rem}{Remark}

\begin{document}
\title{Assouad dimensions of complementary sets}
\author[I. Garcia]{Ignacio Garcia}
\address{Dept. of Pure Mathematics, University of Waterloo, Waterloo, Ont. N3L 2G1,
Canada \\
Depto. de Matem\'atica, Facultad de Ciencias Exactas y Naturales, and
Instituto de Investigaciones F\'isicas de Mar del Plata (CONICET),
Universidad Nacional de Mar del Plata, Argentina }
\email{nacholma@gmail.com}
\author[K. Hare]{Kathryn Hare}
\address{Dept. of Pure Mathematics, University of Waterloo, Waterloo, Ont.  N3L 2G1,
Canada}
\email{kehare@uwaterloo.ca}
\author[F. Mendivil]{Franklin Mendivil}
\address{Dept. of Mathematics and Statistics, Acadia University, Wolfville,
Nova Scotia, Canada, B4P 2R6}
\email{franklin.mendivil@acadiau.ca}
\thanks{The work of I. Garcia was partially supported by a grant from the
Simons Foundation. The work of K. Hare was supported by NSERC 2011-44597. The work of F. Mendivil was supported by NSERC 238549.}
\maketitle

\begin{abstract}
Given a positive, decreasing sequence $a,$ whose sum is $L$, we consider all
the closed subsets of $[0,L]$ such that the lengths of their complementary
open intervals are in one to one correspondence with the sequence $a$. The
aim of this note is to investigate the possible values that Assouad-type
dimensions can attain for this class of sets. In many cases, the set of
attainable values is a closed interval whose endpoints we determine.
\end{abstract}


\section{Introduction}

In this paper, we are interested in the Assouad dimension, introduced in 
\cite{A1, A2} by Assouad, and its `dual', the Lower Assouad dimension,
introduced in \cite{larman1967new} by Larman. Assouad dimensions were
initially of interest because of their application in the theory of
embeddings of metric spaces in Euclidean spaces and in the study of
quasisymmetric maps (see \cite{Heinonen}, \cite{mackay2010conformal} and
references therein). Recently, the Assouad-type dimensions have received
much attention in the fractal geometry community (c.f., \cite%
{chen2016accessible}, \cite{FHORM}, \cite{Fraser}, \cite{fraser2015assouad}, 
\cite{mackay2011assouad}) as they provide quantitative information on the
extreme local behaviour of the geometry of the underlying set. Indeed, if we
denote by $\dim _{A}E$ and $\dim _{L}E$ the Assouad and Lower Assouad
dimensions of the compact set $E$ respectively, then the following relations
hold between these dimensions and the more familiar Hausdorff, packing and
upper box dimensions:%
\begin{equation*}
\dim _{L}E\leq \dim _{H}E\leq \dim _{P}E\leq \overline{\dim }_{B}E\leq \dim
_{A}E\text{.}
\end{equation*}

Given $a=\{a_{j}\},$ a positive, decreasing sequence with finite sum $L$, we
define the class $\mathscr{C}_{a}$ to be the family of all closed subsets of 
$[0,L]$ whose complement (in $[0,L]$) is comprised of disjoint open
intervals with lengths given by the entries of $a$. We call such sets the
complementary sets of $a$. Notice that every compact subset of $\mathbb{R}$
of zero Lebesgue measure belongs to exactly one $\mathscr{C}_{a}$ and every
set in $\mathscr{C}_{a}$ has Lebesgue measure zero. Each family, $\mathscr{C}%
_{a},$ contains both countable and uncountable sets, hence it is natural to
ask about the dimensions of the sets in a given family.

This problem was first studied by Besicovitch and Taylor \cite{BT} for the
case of the Hausdorff dimension. They proved that the set of attained values
for the Hausdorff dimension of sets in $\mathscr{C}_{a}$ was the closed
interval $[0,\dim _{H}C_{a}],$ where $C_{a}$ is the Cantor set associated to
the sequence $a$. (For its definition see section 2.) A similar result was
discovered by Zuberman and the last two authors in \cite{HMZ} in the case of
packing measure. In contrast, it is not difficult to see that all the
complementary sets in $\mathscr{C}_{a}$ have the same lower and upper box
dimension (see \cite{Fal}, Chapter 4).

Here we study this problem for the Assouad-type dimensions. One example of a
countable set in $\mathscr{C}_{a}$ is the decreasing set, denoted $D_{a},$
obtained by placing the gaps in order, beginning at $L$. Although this set
has the smallest cardinality, we show that its Assouad dimension is maximal
among all Assouad dimensions of sets in $\mathscr{C}_{a}$. Further, we prove
that there are only two possible values for $\dim _{A}D_{a}$, namely $0,1$,
and we characterize when these choices occur in terms of properties of the
sequence. The minimum Assouad dimension is shown to be attained by the
associated Cantor set $C_{a}$. The proofs of these results are given in
Subsections \ref{section maximal assouad dimension} and \ref{section minimal
assouad dimension}.

Having found the extreme values, we also consider the structure of the set
of attainable values. We give an example of a sequence $a$ for which each
complementary set has either Assouad dimension $0$ or $1$, and both values
occur. In this example the sequence decreases very rapidly. Under the
assumption of a suitably controlled rate of decay, we prove that given any $%
s\in \lbrack \dim _{A}C_{a},\dim _{A}D_{a}]$ there is some $E\in \mathcal{C}%
_{a}$ such that $\dim _{A}E=s$. These arguments can be found in
Subsection \ref{structure}.

In Section \ref{proofs lower dimension}, we see that in the case of the
Lower Assouad dimension the opposite situation occurs: the Lower Assouad
dimension is minimized by $D_{a}$ and maximized by $C_{a}$. Under a suitable
technical condition on the relative sizes of the $a_{j}$, we again prove
that the set of attainable values of the dimension is a closed interval.

We begin, in Section \ref{prelim}, by introducing our terminology and
notation. There we also derive formulas for the (Lower) Assouad dimensions
of the associated sets $C_{a},$ generalizing the formulas found in \cite%
{li2014assouad} and \cite{ORS} for the special case of central Cantor sets.
These formulas will be very useful for the proofs given later in the paper.

\section{Preliminaries}

\label{prelim}

\subsection{Definitions and Terminology}$\empty$

\textbf{Assouad-type dimensions.} We begin by recalling the definitions of
the Assouad-type dimensions. Given a non-empty set $F\subset \mathbb{R}$ we
denote by $N_{r}(F)$ the minimum number of closed balls of radius $r$ needed to cover $F$.
By $B(x,R)$ we mean the closed ball of radius $R$ centred at $x$.

The \emph{Assouad dimension} of $F$ is defined as 
\begin{align*}
\dim _{A}F=\inf \ \Bigl\{\alpha :\ & \text{there are constants }c,\rho >0\ 
\text{such that } \\
& \sup_{x\in F}N_{r}(B(x,R)\cap F)\leq c\left( \frac{R}{r}\right) ^{\alpha
}\ \ \text{for all }0<r<R<\rho \Bigr\}.
\end{align*}%
Dually, the \emph{Lower Assouad dimension} of $F$ is defined as 
\begin{align*}
\dim _{L}F=\sup \ \Bigl\{\alpha :\ & \text{there are constants }c,\rho >0\ 
\text{such that } \\
& \inf_{x\in F}N_{r}(B(x,R)\cap F)\geq c\left( \frac{R}{r}\right) ^{\alpha
}\ \ \text{for all }0<r<R<\rho \Bigr\}.
\end{align*}

Lower Assouad dimension has simply been called Lower dimension in the
literature. We are calling it Lower Assouad dimension to emphasize the
relationship with the Assouad dimension.

For a precompact set $F$ we have 
\begin{equation*}
\dim _{L}F\leq \underline{\dim }_{B}F\leq \overline{\dim }_{B}F\leq \dim
_{A}F\text{,}
\end{equation*}%
while if $F$ is compact then 
\begin{equation*}
\dim _{L}F\leq \dim _{H}F
\end{equation*}%
(see \cite{larman1967new, larman1967hausdorff}). The Assouad-type dimensions
are quite sensitive to the local structure of the set. Indeed, if $F$ has an
isolated point then $\dim _{L}F=0$. If $F$ is a self-similar set, then the
Lower Assouad and Hausdorff dimensions coincide, and these coincide with the
Assouad dimension if $F$ satisfies the open set condition or if $F\subseteq 
\mathbb{R}$ satisfies the weak separation condition. For these facts and for
further background information on these dimensions we refer to the papers of
Fraser (et al) \cite{FHORM}, \cite{Fraser}, Luukkainen \cite%
{luukkainen1998assouad} and Olson \cite{olson2002bouligand}.

\medskip
	
\textbf{Complementary sets.} Let $a=\{a_{j}\}$ be a positive sequence with
finite sum $L$ and let $I=[0,L]$. Our interest is in the family $\mathscr{C}%
_{a}$ which consists of the closed subsets $E$ of $I,$ of the form $%
E=I\setminus \bigcup_{j=1}^{\infty }U_{j}$, where $\{U_{j}\}$ is a disjoint
family of open intervals contained in $I$, with the length of $U_{j}$ equal
to $a_{j}$ for each \thinspace $j$. The members of $\mathscr{C}_{a}$ are
called the \emph{complementary sets of }$a$. One complementary set is the
countable decreasing (from right to left) set $D_{a}=\{\sum_{i=k}^{\infty
}a_{i}\}_{k}.$

We remark that if $b$ is  any rearrangement of $a$, then the complementary sets are the same; in particular, this is true if $b$ is the decreasing rearrangement.

\medskip

\textbf{Associated Cantor sets.} Another set in each class $\mathscr{C}_{a}$
is what we call the \emph{associated Cantor set} and denote by $C_{a}.$ It
is constructed as follows: In the first step, we remove from $I$ an open
interval of length $a_{1}$, resulting in two closed intervals $I_{1}^{1}$
and $I_{2}^{1}$. Having constructed the $k$-th step, we obtain the closed
intervals $I_{1}^{k},\ldots ,I_{2^{k}}^{k}$ contained in $I$. The next step
consists in removing from each $I_{j}^{k}$ an open interval of length $%
a_{2^{k}+j-1}$, obtaining the closed intervals $I_{2j-1}^{k+1}$ and $%
I_{2j}^{k+1}$. We define $C_{a}:=\bigcap_{k\geq
1}\bigcup_{j=1}^{2^{k}}I_{j}^{k}$. This construction uniquely determines an
uncountable set that is homeomorphic to the classical middle-third Cantor
set. Indeed, the classical Cantor set is the Cantor set associated with the
sequence $\{a_{i}\}$ where $a_{i}=3^{-n}$ if $2^{n-1}\leq i\leq 2^{n}-1$. We
call $a=\{a_{j}\}$ \emph{the gap sequence} of $C_{a}$.

In the case when the sequence $a$ is decreasing (i.e., $a_{j}\geq a_{j+1}$
for all $j$), we will call $C_{a}$ a \emph{decreasing Cantor set}. The
classical Cantor set is such an example.

\medskip

\textbf{Central Cantor sets.} Given a sequence of ratios $\{r_{j}\}$ with $%
0<r_{j}<1/2$, we construct a \emph{central Cantor set }$C{\{r_{j}\}}$\emph{\ 
}in the same fashion as the classical Cantor set, but with the $2^{k}$
intervals of step $k$ having lengths $r_{1}\cdot \cdot \cdot r_{k}$. This
set is the Cantor set, $C_{a},$ associated to the gap sequence $a=\{a_{i}\}$
where $a_{1}=1-2r_{1}$ and $a_{i}=r_{1}\cdot \cdot \cdot r_{n}(1-2r_{n+1})$
if $2^{n}\leq i\leq 2^{n+1}-1$. Note that this gap sequence is not
necessarily decreasing, but rather lists the gaps ordered by level and
repeated according to their multiplicity.

For such a choice of gap sequence $a$, we always have $\dim _{A}D_{a}=1$.
This is because for each $k$ there are $2^{k}$ gaps all of the same length,
say, $g_{k}$. If we choose $R=2^{k}g_{k}$ and $r=g_{k},$ then it is easy to
see that for a suitable point $x$, $N_{r}(B(x,R)\cap D_{a})=2^{k}=R/r$.

\subsection{Corresponding complementary sets}

Given two summable sequences, $a=\{a_{n}\}$ and $b=\{b_{n}\}$, one can
assign a natural bijection between elements in $\mathscr C_{a}$ and $%
\mathscr
C_{b}$ as follows (see for example \cite{CMMS}, Sections 2 and 3 for the
case of Cantor sets). Let $E\in \mathscr C_{a}$ and denote by $g_{n}^{(a)}$
the complementary gap of $E$ of length $a_{n}$. For $x\in E$, let $%
L_{a}(x)=L_{a}(x,E)=\bigl\{n\in \mathbb{N}:g_{n}^{(a)}\subset \lbrack 0,x]\bigr\}$, so
that 
\begin{equation*}
x=\sum_{n\in L_{a}(x)}a_{n}.
\end{equation*}%
We define 
\begin{equation*}
\pi (x)=\sum_{n\in L_{a}(x)}b_{n}.
\end{equation*}%
Notice that $\pi :E\rightarrow \pi (E)\subset \lbrack 0,\sum b_{i}]$ is an
increasing homeomorphism and that $\pi (E)\in \mathscr C_{b}$; the last part
of the observation follows by letting $g_{n}^{(a)}=(x,y)$, so that $(\pi
(x),\pi (y))$ is the complementary gap in $F$ of length $b_{n}$. Moreover,
for any $F\in \mathscr C_{b}$ we can do the inverse process to obtain $%
L_{b}(z,F)=L_{a}(x,E)$, where $x=\psi (z)$, $E=\psi (F)$ and $\psi $ the
corresponding homeomorphism. By the above considerations, the map $\Pi :%
\mathscr C_{a}\rightarrow \mathscr C_{b}$ given by $\Pi (E)=\pi (E)$ is a
bijection. We say that $E\in \mathscr C_{a}$ and $F\in \mathscr C_{b}$ are 
\emph{corresponding complementary sets} if $\Pi (E)=F$.

Two sequences $\{a_{n}\}$ and $\{b_{n}\}$ are said to be \emph{equivalent}
if $a_n\sim b_n$, that is, if there is $c\geq 1$ such that 
\begin{equation*}
c^{-1}\leq \frac{a_{n}}{b_{n}}\leq c\ \ \ \text{for all}\ n.
\end{equation*}%
Recall that Assouad-type dimensions are bi-Lipschitz invariant; see \cite%
{Fraser}.

\begin{proposition}
\label{equivalentarrangement} If $a$ and $b$ are equivalent sequences, then
the corresponding complementary sets in $\mathscr C_{a}$ and $\mathscr C_{b}$
are bi-Lipschitz equivalent and hence have the same (Lower) Assouad
dimension.
\end{proposition}

\begin{proof}
If $x<y$, where $x,y\in E$ and $E\in \mathscr C_a$, then by definition of $%
\pi$ and the equivalence hypothesis, 
\begin{equation*}
\pi(y)-\pi(x)=\sum_{n\in L_a(y)\setminus L_a(x)} b_n \ \sim \sum_{n\in
L_a(y)\setminus L_a(x)} a_n=y-x.
\end{equation*}
This proves the statement.
\end{proof}

We say a decreasing sequence $\{a_{n}\}$ is \emph{doubling }if there is a
constant $\tau $ such that $a_{n}\leq \tau a_{2n}$ for all \thinspace $n$.
If $a$ is a decreasing, doubling sequence with $\sum_{j}a_{j}=1$, then it is
easy to see that $a$ is equivalent to the sequence of gaps of a central
Cantor set with ratios $r_{k}$ given by $1-2r_1 =a_1$ and 
\begin{equation}
r_{1}\cdot \cdot \cdot r_{n}(1-2r_{n+1})=2^{-n}(a_{2^{n}}+\cdot \cdot \cdot
+a_{2^{n+1}-1})  \label{centralCantor}
\end{equation}%
for $n \geq 1$. Consequently, the Assouad-type dimensions of $C_{a}$
coincide with those of this central Cantor set.

\subsection{Assouad type dimensions of $\mathbf{C_\emph{a}}$}

\label{Assouad type dimensions of $C_a$}

Given a decreasing, summable sequence $a=\{a_{j}\},$ we put 
\begin{equation*}
s_{n}^{(a)}=2^{-n}\sum_{j=2^{n}}^{\infty }a_{j},\text{ for }n\geq 0.
\end{equation*}%
When the sequence $a$ is clear we simply write $s_{n}$. This is the average
length of the intervals that arise at step $n$ in the construction of the
associated decreasing Cantor set. In the special case that $a$ is the
sequence of gaps of a central Cantor set $C\{r_{j}\}$, then $s_{n}$ is the
length of the step $n$ Cantor intervals, $r_{1}\cdot \cdot \cdot r_{n}$.
More generally, if $I_{j}^{n}$ is any interval arising in step $n$ of the
decreasing Cantor set $C_{a}$, then $s_{n+1}\leq $ $\left\vert
I_{j}^{n}\right\vert \leq s_{n-1}$.

Conversely, given $a=\{a_{j}\}$ with $\sum a_{j}=1$, if we define $r_{n}$ by
(\ref{centralCantor}), then 
\begin{equation*}
r_{1}\cdot \cdot \cdot r_{n}(1-2r_{n+1})=s^{(a)}_{n}-2s^{(a)}_{n+1}
\end{equation*}%
and an easy induction argument shows that $s^{(a)}_{n}=r_{1}\cdot \cdot
\cdot r_{n} $.

We can calculate the Assouad-type dimensions for the sets $C_{a}$ in terms
of the sequence $\{s^{(a)}_{n}\}$.

\begin{theorem}
Let $a$ be a decreasing, summable sequence and $C_{a}$ the associated Cantor
set. Then%
\begin{align}
\dim _{A}C_{a}=\inf :\Bigl\{\beta :\exists \ k_{\beta },\ n_{\beta }\ \text{%
such that}\ \forall & k\geq k_{\beta },n\geq n_{\beta }  \label{AssouadDec}
\\
\frac{n\log 2}{\log s_{k}/s_{k+n}}& \leq \beta \Bigr\}.  \notag
\end{align}%
and%
\begin{align}
\dim _{L}C_{a}=\sup :\Bigl\{\beta :\exists \ k_{\beta },\ n_{\beta }\ \text{%
such that}\ \forall & k\geq k_{\beta },n\geq n_{\beta }  \label{LowerDim} \\
\frac{n\log 2}{\log s_{k}/s_{k+n}}& \geq \beta \Bigr\}.  \notag
\end{align}
\end{theorem}

\begin{proof}
Assouad Dimension of $C_a$: Let $d=\dim _{A}C_a$ and $\alpha $ equal the
right hand side in (\ref{AssouadDec}). Given $\epsilon >0$ there are
constants $c$ and $\rho $ (depending on $\epsilon )$ such that if $%
0<r<R<\rho $, then 
\begin{equation*}
N_{r}(B(x,R)\cap C_a)\leq c\left( \frac{R}{r}\right) ^{d+\epsilon }.
\end{equation*}%
Pick $k\geq k_{0}$ so large that $s_{k}<\rho ,$ and let $I^{k+1}_j$ be any
interval of step $k+1$ in the construction. Put $R=\left\vert
I^{k+1}_j\right\vert $ and $r=s_{k+n}$ so 
\begin{equation*}
\frac{R}{r}=\frac{\left\vert I^k_j\right\vert }{s_{k+n}}\leq \frac{s_{k}}{%
s_{k+n}}.
\end{equation*}%
Let $x$ denote the left endpoint of interval $I^{k+1}_j$. As the intervals
in the construction at level $k+n-1$  have length at least $r$ and $%
B(x,R)\supseteq I^{k+1}_j$, it follows that $N_{r}(B(x,R))\geq 2^{n-2}$.
Thus 
\begin{equation*}
2^{n-2}\leq c\left( \frac{s_{k}}{s_{k+n}}\right) ^{d+\epsilon }\ \ \text{for
all }k\geq k_{0}\ \text{and }n\geq 2.
\end{equation*}%
But $\log s_{k}/s_{k+n}\geq n\log 2$, so 
\begin{align*}
d+\epsilon & \geq \frac{n\log 2}{\log s_{k}/s_{k+n}}-\frac{\log 4c}{\log
s_{k}/s_{k+n}} \\
& \geq \frac{n\log 2}{\log s_{k}/s_{k+n}}-\frac{\log 4c}{n\log 2}.
\end{align*}%
If we choose $n_{0}>1$ such that $\log 4c/(n\log 2)<\epsilon $ for all $%
n\geq n_{0}$ then, 
\begin{equation*}
d+2\epsilon \geq \frac{n\log 2}{\log s_{k}/s_{k+n}}\text{ for all }k\geq
k_{0},n\geq n_{0}.
\end{equation*}%
It follows that $\alpha \leq d$.

To see that $d\leq \alpha $ we will show that for any $\epsilon >0$ there
are $c$, $\rho $ such that if $0<r<R<\rho $, then 
\begin{equation}
N_{r}(B(x,R)\cap C_a)\leq c\left( \frac{R}{r}\right) ^{\alpha +\epsilon }.
\label{eq4}
\end{equation}%
Choose $k_{\epsilon }$, $n_{\epsilon }$ as in the definition of $\alpha $,
so 
\begin{equation*}
\frac{n\log 2}{\log s_{k}/s_{k+n}}\leq \alpha +\epsilon \text{ for all }%
k\geq k_{\epsilon },\ n\geq n_{\epsilon }.
\end{equation*}%
Put $\rho =s_{k_{\varepsilon }}/2$. Consider any $0<r<R<\rho $ and $B(x,R)$
with $x\in C_a$. Choose $k,n$ such that 
\begin{equation*}
s_{k}\leq 2R<s_{k-1}\ \ \ \text{(so}\ k\geq k_{\epsilon }),
\end{equation*}%
and 
\begin{equation*}
s_{n+1}\leq r<s_{n}\ \ \text{(so }n\geq k-1)).
\end{equation*}%
If $B(x,R)$ intersects (non-trivially) at least five intervals of step $k-1$
in $C_a$, then it will fully contain at least one of level $k-2$. But this
is impossible since $\text{diam}\: B(x,R)<s_{k-1}$ and this is less than the
length of any interval of step $k-2$. Thus $B(x,R)$ intersects at most four
step $k-1$ intervals. The (closed) balls of radius $r$ centred at the left
endpoints of the step $n+2$ intervals contained in these step $k-1$
intervals cover $B(x,R)\cap C_a$ and hence $N_{r}(B(x,R))\leq 4
\cdot2^{n-k+3}$. Furthermore, 
\begin{equation*}
\frac{R}{r}\geq \frac{s_{k}/2}{s_{n}}.
\end{equation*}%
If $n=k+m$ for $m\geq 1$ and $m\geq n_{\epsilon },$ then $2^{m}\leq
(s_{k}/s_{k+m})^{\alpha +\epsilon }$, so 
\begin{equation*}
N_{r}(B(x,R)\cap C_a)\leq 32\cdot 2^{m}\leq 32\left( \frac{s_{k}}{s_{k+m}}%
\right) ^{\alpha +\epsilon }\leq c\left( \frac{R}{r}\right) ^{\alpha
+\epsilon }
\end{equation*}%
for a suitable constant $c$.

Since $N_{r}(B(x,R)\cap C_a)\leq 32 \cdot 2^{n_{\varepsilon }}$ if $m\leq
n_{\varepsilon }$ and is at most $32$ if $n=k-1$ or $k$, and $R/r\geq 1,$ it
follows that by making a larger choice of constant, if necessary, 
inequality (\ref{eq4}) will hold for $1\leq m\leq n_{\epsilon }$ and for $n=k-1,k$.
Since (\ref{eq4}) is satisfied, $d\leq \alpha $ and hence we have $\dim
_{A}C_a=\alpha $ .

\medskip Lower Assouad Dimension of $C_a$: In the case when $\inf
s_{k+1}/s_{k}>0$, the argument that the Lower Assouad dimension of $C_a$ is
equal to the right side of (\ref{LowerDim}) is similar to the argument for
the Assouad dimension. The extra assumption is used to find an upper bound
on $R/r$ when $1\leq m\leq n_{\varepsilon } $ or $n=k-1,k$ in the final step
of the proof.

So suppose $\inf s_{k+1}/s_{k}=0$. We will show that both $\dim _{L}C_{a}=0$
and the RHS of (\ref{LowerDim}) $=0$ and hence the two are equal.

It is easy to see that $\dim _{L}C_a=0$ if $\inf s_{k+1}/s_{k}=0$: just take 
$R=s_{k}/2$ and $r=s_{k+1}/2$. Then any ball of radius $R$ intersects at
most four intervals of step $k$ in $C_a$, for otherwise, it would fully
contain a step $k-1$ interval and those have length at least $s_{k}\geq 
\text{diam}\: B(x,R)$. Consequently, $B(x,R)$ can be covered by at most 16
balls of radius $r$, namely balls centred at the endpoints of the step $k+2$
intervals contained in four step $k$ intervals. As we can make such choices $%
R,r$ with $R/r\rightarrow \infty $ it follows that we cannot satisfy $%
N_{r}(B(x,R) \cap C_a)\geq c(R/r)^{\varepsilon }$ for any $\varepsilon >0$.

Now assume the RHS of (\ref{LowerDim}) is strictly positive. Then there will
be some $\gamma >0$ such that 
\begin{equation*}
\left( \frac{s_{k+n}}{s_{k}}\right) ^{1/n}\geq 2^{-\gamma }
\end{equation*}%
for all $k,n$ sufficiently large, say $k\geq K$ and $n\geq N$. In
particular, for all $k\geq K$ 
\begin{equation*}
2^{-\gamma /N}\leq \frac{s_{k+N}}{s_{k}} \leq \frac{s_{k+1}}{s_{k}}\text{.}
\end{equation*}%
As $N$ is fixed, it follows that $\inf s_{k+1}/s_{k}>0$, a contradiction.
\end{proof}

\begin{corollary}
\label{CoroDimLzero} (i) $\dim _{L}C_{a}=0$ if and only if $\inf
s_{k+1}/s_{k}=0$. In particular, if $a$ is a doubling sequence, then $\dim
_{L}C_{a}>0$.

(ii) $\dim _{L}C\{r_{j}\}=0$ if and only if $\inf r_{j}=0$.
\end{corollary}

\begin{proof}
(i) The first statement follows from the proof of the theorem. For the
second statement we note that the doubling assumption implies $\inf
s_{k+1}/s_{k}>0$. This is because under this assumption 
\begin{equation*}
2^{-n}\sum_{j=2^{n}+1}^{2^{n+1}}a_{j}\sim a_{2^{n}}\text{ }
\end{equation*}%
and thus there are positive constants $c_1, c_2$ so that $s_{n-1}\leq
c_1(a_{2^{n}}+s_{n}) $, while $s_{n}\geq c_2a_{2^{n}}$.

(ii) is immediate since $r_{j}=s_{j}/s_{j-1}$.
\end{proof}

\begin{remark}
\label{DimComment}It is an easy exercise to check that we have%
\begin{align}
\dim _{A}C_{a}& =\inf :\Bigl\{\beta :\exists \ k_{\beta },\ n_{\beta }\ 
\text{such that}\ 2^{n}\leq \left( \frac{s_{k}}{s_{k+n}}\right) ^{\beta }%
\text{ }\forall k\geq k_{\beta },n\geq n_{\beta }\Bigr\}
\label{formula upper cantor} \\
& =\limsup_{n}\left( \sup_{k}\frac{n\log 2}{\log (s_{k}/s_{k+n})}\right)
\label{AdimAlt}
\end{align}%
and 
\begin{align}
\dim _{L}C_{a}& =\sup :\Bigl\{\beta :\exists \ k_{\beta },\ n_{\beta }\ 
\text{such that }2^{n}\geq \left( \frac{s_{k}}{s_{k+n}}\right) ^{\beta }\ 
\text{ }\forall k\geq k_{\beta },n\geq n_{\beta }\Bigr\}
\label{formula lower} \\
& =\liminf_{n}\left( \inf_{k}\frac{n\log 2}{\log (s_{k}/s_{k+n})}\right).
\label{LowerDimAlt}
\end{align}
\end{remark}


\begin{remark}
Similar arguments show that the same formula holds if $C\{r_{j}\}$ is a
central Cantor set. Previously, the Assouad dimension formula (\ref{AdimAlt}%
) was obtained for central Cantor sets, using other methods, by Olson et al.
in \cite{ORS}; see also Li et al. \cite{li2014assouad}.
\end{remark}

\section{The Assouad dimensions of complementary sets}

\label{proofs}

In this section we will study the Assouad dimensions of the sets that belong
to a family $\mathscr{C}_{a}$. Note that throughout the paper the constant $c>0$ appearing in inequalities may change from one occurrence to another.

\subsection{Maximal Assouad dimension}

\label{section maximal assouad dimension} We first show that the decreasing
rearrangement, $D_{a}$, gives the largest Assouad dimension. We prove that
its value is either $0$ or $1$ and we characterize when it is $0$ in terms
of a lacunary type condition.

\begin{proposition}
\label{01fordecreasing} Let $a=\{a_{j}\}$ be a decreasing, summable sequence
and let $D_{a}\in \mathscr{C}_{a}$ be the decreasing rearrangement. Then $%
\dim _{A}D_{a}=0$ or $1$. Moreover, $\dim _{A}D_{a}=0$ if and only if there
is some $\varepsilon >0$ such that $a_{j}\geq \varepsilon
\sum_{i=j+1}^{\infty }a_{i}$ for all $j$.
\end{proposition}

\begin{proof}
Let $x_{j}=\sum_{i=j}^{\infty }a_{i}$, so that $D_{a}=\{x_{j}\}$. Note that $%
a_{j}=x_{j}-x_{j+1}$ and that $a_{j}\geq \varepsilon \sum_{i=j+1}^{\infty
}a_{i}$ if and only if $a_{j}\geq \varepsilon x_{j+1}$ if and only if $%
x_{j}\geq (1+\varepsilon )x_{j+1}.$

First, we will assume there is no $\varepsilon >0$ such that $a_{j}\geq
\varepsilon \sum_{i=j+1}^{\infty }a_{i}$ for all $j$ and prove $\dim
_{A}D_{a}=1$. Under this assumption, for every positive integer $N_{0}$
there are arbitrarily large $k$ such that $a_{k}/x_{k+1}<1/N_{0}$. Given any
such $k$, choose $N\geq N_{0}$, depending on $k$, such that 
\begin{equation}
\frac{1}{N+1}\leq \frac{a_{k}}{x_{k+1}}<\frac{1}{N}.  \label{eq}
\end{equation}%
Let $R=x_{k+1}$ and $r=a_{k}$. Note that (\ref{eq}) implies 
\begin{equation*}
\frac{R}{r}\leq N+1\text{ and }\frac{x_{k+1}}{N+1}\leq r<\frac{x_{k+1}}{N}<R.
\end{equation*}

Consider $B(0,R)\cap D_{a}=\{x_{j}:j\geq k+1\}$. As $x_{k+1}>Nr$, the
intervals 
\begin{equation*}
\lbrack x_{k+1}-(s+1)r,x_{k+1}-sr)\ \ \ \text{for }s=0,\ldots ,N-1
\end{equation*}%
are contained in $[0,x_{k+1}]$. Each has length $r$ and as $r\geq a_{i}$ for 
$i\geq k$ by the decreasing assumption, each of these intervals must contain
some $x_{j}$. Consequently, 
\begin{equation*}
N_{r/2}(B(0,R)\cap D_{a})\geq \left\lfloor \frac{N}{2}\right\rfloor .
\end{equation*}%
Since we can choose $N$ arbitrarily large and $R$ arbitrarily small, the
fact that $R/r\leq N+1$ forces $\dim _{A}D_{a}=1$.

Next, assume there is some $\varepsilon >0$ such that $a_{j}\geq \varepsilon
\sum_{i=j+1}^{\infty }a_{i}=\varepsilon x_{j+1}$ for all $j$. Putting $%
\lambda =1+\varepsilon $ we have $x_{k}/x_{k+1}\geq \lambda $ for all $k$.
Choose any $0<r<R$ and consider $B(x,R)$ for $x\in D_{a}$.

\noindent \emph{Case 1: $x\leq R$.} Then $B(x,R)\cap D_{a}=[0,x+R]\cap D_{a}$%
. Choose the minimal index $k$ such that $x_{k}\leq x+R$.

\noindent \emph{a)} Assume there is an $i\geq k+1$ such that $a_{i}<r\leq
a_{i-1}$. Then $r>\varepsilon x_{i+1}$ so $[0,x_{i+1}]\cap D_{a}$ is covered
by $1/\varepsilon $ balls of radius $r$. On the other hand, $[x_{i},x+R]\cap
D_{a}=\{x_{j}\}_{j=k}^{i}$, hence 
\begin{equation*}
N_{r}(B(x,R)\cap D_{a})\leq i-k+1+1/\varepsilon .
\end{equation*}%
Since $[0,x_{k}]\subset B(x,R)$, $R\geq x_{k}/2$ and by the lacunarity
assumption $x_{k}\geq \lambda ^{i-k-1}x_{i-1}$. Furthermore, $r\leq x_{i-1},$
so $R/r\geq \lambda ^{i-k-1}/2.$ Since $\lambda >1,$ it follows that given
any $\alpha >0$ there is a constant $c$ depending on $\alpha $ such that 
\begin{equation*}
N_{r}(B(x,R)\cap D_{a})\leq c\left( \frac{\lambda ^{i-k-1}}{2}\right)
^{\alpha }\leq c\left( \frac{R}{r}\right) ^{\alpha }.
\end{equation*}

\noindent \emph{b)} Otherwise, $r\geq a_{k}\geq \varepsilon x_{k+1}$, so $%
N_{r}(B(x,R)\cap D_{a})\leq 1+1/\varepsilon$. Since $R/r\geq 1$, we clearly
have $N_{r}(B(x,R)\cap D_{a})\leq c\left( R/r\right) ^{\alpha }$ for a
suitable constant $c$.

\noindent \emph{Case 2: $x>R$.} Under this assumption there are indices $%
K\leq J$ such that $B(x,R)=\{x_{i}\}_{i=K}^{J}$ and thus $2R\geq
x_{K}-x_{J}\geq (\lambda ^{J-K}-1)x_{J}$. It follows easily from these
observations that if $a_{K} \le r$ or $r \le a_{J}$, then for any $\alpha >0$
there is a constant $c$ depending on $\alpha$ such that $N_{r}(B(x,R)\cap
D_{a})\leq c(R/r)^{\alpha }$.

Thus we can assume there is some $i$ with $K\leq i-1\leq J$ such that $%
\varepsilon x_{i+1}\leq a_{i}<r\leq a_{i-1}\leq x_{i-1}$. As before, $%
N_{r}(B(x,R)\cap D_{a})\leq i-K+1+1/\varepsilon $. Now $x_{K}\geq
x_{i-1}\geq x_{J},$ thus the lacunarity property implies 
\begin{equation*}
2R\geq x_{K}-x_{J}\geq x_{K}-x_{i-1}\geq (\lambda ^{i-K-1}-1)x_{i-1}.
\end{equation*}
Hence 
\begin{equation*}
\frac{R}{r}\geq \frac{1}{2}(\lambda ^{i-k-1}-1)
\end{equation*}%
and, as in the first case, it follows that given any $\alpha >0$, $%
N_{r}(B(x,R)\cap D_{a})\leq c\left( R/r\right) ^{\alpha }$ for a suitable
constant $c$.

Since we can obtain this bound for arbitrary $\alpha >0$, we deduce that $%
\dim_A D_a =0$.
\end{proof}

\begin{proposition}
\label{lacunary} Suppose $a=\{a_{k}\}$ is a decreasing, summable sequence
and that $\dim _{A}D_{a}=0$. If $E\in \mathscr{C}_{a}$, then $\dim _{A}E=0$.
\end{proposition}

\begin{proof}
As $\dim _{A}D_{a}=0$, the previous proposition implies there is some $%
\varepsilon >0$ such that $a_{j}\geq \varepsilon \sum_{i=j+1}^{\infty }a_{i}$
for all $j$. Choose any $0<r<R\leq a_{1}$ and let $x\in $ $E$.

Pick $k$ such that $a_{k}\leq r<a_{k-1}$. Let $n$ be the total number of
gaps of length $a_{l}$ for some $l<k$ that are contained in $B(x,R)$. The
length of the interval between any two such consecutive gaps is at most 
\begin{equation*}
a_k + \sum_{i=k+1}^{\infty }a_{i}\leq \left(1+\frac{1}{\varepsilon }\right)a_{k}\leq
r\left(1+ \frac{1}{\varepsilon }\right).
\end{equation*}%
It follows that $N_{r}(B(x,R)\cap E)\leq n(1+1/\varepsilon ).$ If the $n$
gaps of length $a_{l}$, $l<k$, that are contained in $B(x,R)$ have lengths $%
a_{l_{1}}, \dots,a_{l_{n}}$ where $l_{j}<l_{j+1}$, then the lacunarity
assumption implies 
\begin{align*}
R& \geq a_{l_{1}}+\sum_{j=2}^{n}a_{l_{j}}\geq \varepsilon
\sum_{j=l_{1}+1}^{\infty }a_{j}+\sum_{j=2}^{n}a_{l_{j}} \\
& \geq (1+\varepsilon )\sum_{j=2}^{n}a_{l_{j}}\geq (1+\epsilon)
^{n-1}a_{l_{n}}\geq (1+\epsilon) ^{n-1}a_{k-1}.
\end{align*}%
Hence $R/r\geq (1+\epsilon) ^{n-1}$ for all $n$. This proves that $\dim
_{A}E=0$.
\end{proof}

\begin{corollary}
\label{coromax} If $a$ is any summable, decreasing sequence and $E\in %
\mathscr{C}_{a}$ is any complementary set, then $\dim _{A}E\leq \dim
_{A}D_{a}$.
\end{corollary}

\begin{proof}
If $\dim_A D_a=1$ it is clearly maximal. Otherwise it is $0$ and then all
other complementary sets also have Assouad dimension $0$ by the Proposition.
\end{proof}

\begin{corollary}
Suppose $\{a_j\}$ is a positive, decreasing sequence with the property that
there are positive constants $c_1, c_2$ and $\lambda <1$ such that $%
c_1\lambda^j \leq a_j \leq c_2\lambda^j$. If $E \in \mathscr{C}_{a}$, then $%
\dim_AE=0$.
\end{corollary}

\begin{proof}
By Proposition \ref{01fordecreasing}, $\dim_A D_a=0$. Now appeal to the
previous Corollary.
\end{proof}

\subsection{Minimal Assouad dimension}

\label{section minimal assouad dimension}

Next, we prove that the decreasing Cantor set has the minimal Assouad
dimension from the family $\mathscr{C}_{a}$.

\begin{theorem}
\label{LowerBdDec}If $a$ is any summable, decreasing sequence, then $\dim
_{A}E\geq \dim _{A}C_a$ for any complementary set $E\in \mathscr{C}_{a}$.
\end{theorem}

\begin{proof}
We can assume $\dim_A C_a >0$, else there is nothing to prove. The proof is
divided into two cases depending on whether or not there are frequently gaps
that are not `too small'.

First assume there is some $\delta >0$ and a constant $N_{0}$ such that
there is no string of $N_{0}$ consecutive ratios, $s_{j}/s_{j-1}$, all
exceeding $1/2-\delta $. Let $\gamma =\dim _{A}C_a$, fix $0<\epsilon <\gamma
/3$ and suppose $E\in \mathscr{C}_{a}$ has $\dim _{A}E=\alpha <\gamma
-3\varepsilon $. Indeed, assume%
\begin{equation}
N_{r}(B(x,R)\cap E)\leq c_{E}\left( \frac{R}{r}\right) ^{\alpha +\epsilon }\
\ \forall \ 0<r < R<\rho .  \label{DimE}
\end{equation}

Pick $k_{0}$ so large that $2^{k}s_{k}<\rho $ if $k\geq k_{0}$. By (\ref%
{formula upper cantor}) we can pick $k$ and $n$ arbitrarily large such that 
\begin{equation*}
2^n\geq \left( \frac{s_{k}}{s_{k+n}}\right) ^{\gamma -\epsilon } \geq
2^{n\epsilon }\left( \frac{s_{k}}{s_{k+n}}\right) ^{\gamma -2\epsilon }.
\end{equation*}

By hypothesis, there is an index $j$ with $k+n-N_{0}<j\leq k+n$ and $%
s_{j+1}/s_{j}<1/2-\delta $. This ensures that 
\begin{equation}
2\delta s_{j}\leq s_{j}-2s_{j+1}=2^{-j}(a_{2^{j}}+\cdot \cdot \cdot
+a_{2^{j+1}-1})\leq a_{2^{j}}  \label{Gapsize}
\end{equation}%
and therefore 
\begin{equation}
2^{n}\geq 2^{n\epsilon }\left( \frac{s_{k}}{s_{j}}\right) ^{\gamma
-2\epsilon }\geq c2^{n\epsilon }\left( \frac{s_{k}}{a_{2^{j}}}\right)
^{\gamma -2\epsilon }  \label{Abound}
\end{equation}%
for a suitable constant $c$.

Removing from $[0,L]$ the complementary gaps of $E$ that have lengths $%
a_{1},...,a_{2^{k}-1}$ (i.e., the gaps of levels $1, \ldots,k$), we obtain
the set $J_{1}\cup \ldots \cup J_{M_{k}}\cup \{\text{singletons}\}$, where $%
J_{i} $ are non-trivial closed intervals and $M_{k} < 2^{k}$. Note that 
\begin{equation*}
\sum_{i=1}^{M_{k}}|J_{i}|=2^{k}s_{k}
\end{equation*}%
as this is the sum of the remaining gaps to place.

Let $b_{i}$ be the number of gaps of level $j+1$ contained in $J_{i}$. Thus
if $x$ is an endpoint of $J_{i}$ and $r=a_{2^j}/2$ we have 
\begin{equation*}
N_{r}(B(x,|J_{i}|)\cap E)\geq b_i.
\end{equation*}%
Moreover, $\sum b_i=2^{j}$, so combining (\ref{Abound}) with H\"{o}lder's
inequality gives 
\begin{align*}
\sum b_{i}& \geq c 2^{j-n}2^{n\epsilon }\left( \frac{s_{k}}{a_{2^{j}}}\right)
^{\gamma -2\epsilon } \\
& \geq c 2^{k-N_{0}}2^{n\epsilon }\left( \frac{2^{-k}\sum_{i=1}^{M_{k}}|J_{i}|%
}{r}\right) ^{\gamma -2\epsilon } \\
& \geq c 2^{n\epsilon }\left( \frac{2^k}{M_{k}}\right) ^{(1-(\gamma
-2\epsilon ))}\frac{\sum |J_{i}|^{\gamma -2\epsilon }}{r^{\gamma -2\epsilon }%
} \\
& \geq c 2^{n\epsilon }\frac{\sum |J_{i}|^{\gamma -2\epsilon }}{r^{\gamma
-2\epsilon }}.
\end{align*}%
Therefore, there exists $i$ such that 
\begin{equation*}
b_{i}\geq c2^{n\epsilon }\left( \frac{|J_{i}|}{r}\right) ^{\gamma -2\epsilon
},
\end{equation*}%
and hence%
\begin{equation*}
N_{r}(B(x,|J_{i}|)\cap E)\geq c2^{n\epsilon }\left( \frac{|J_{i}|}{r}\right)
^{\gamma -2\epsilon }.
\end{equation*}%
As $|J_{i}|>0$ we have $b_{i}\geq 1$ so $J_{i}$ contains some gap of level $j+1$.
This guarantees that $r\leq \left\vert J_{i}\right\vert $. Furthermore, $%
\left\vert J_{i}\right\vert \leq $ $2^{k}s_{k}<\rho $. As $k$ and $n$ can be
made arbitrarily large that contradicts (\ref{DimE}).

Now assume that for any $\delta >0$ and integer $N$, there is a string of at
least $N$ consecutive ratios $s_{j}/s_{j-1}>1/2-\delta =\lambda $. In this
case, we have $s_{k}/s_{k+N}<\lambda ^{-N}$ for some $k$ and hence 
\begin{equation*}
\sup_{k}\left( \frac{N\log 2}{\log s_{k}/s_{k+N}}\right) \geq \frac{\log 2}{%
|\log \lambda |}.
\end{equation*}%
As we can do this for all $N$, it follows from (\ref{AdimAlt}) that $\dim
_{A}C_a\geq \log 2/|\log \lambda |$. Since $\lambda $ can be chosen
arbitrarily close to $1/2$, we see that $\dim _{A}C_a=1$. Thus we need to
show that any complementary set $E$ also has Assouad dimension one.

Note that if $\overline{\dim }_{B} C_a=1,$ then $\overline{\dim }_{B}E=1$
because of the invariance of the upper box dimension for complementary sets.
Hence also $\dim _{A}E=1$ and we are done.

Otherwise, 
\begin{equation}
\overline{\dim }_{B}C_a=\limsup_{n\rightarrow \infty }\frac{n\log 2}{|\log
s_{n}|}<1.  \label{ofrmuladbs}
\end{equation}%
(see \cite{GMS}), hence there must exist some $\lambda _{0}<1/2$ such that $%
s_{j+1}/s_{j}\leq \lambda _{0}$ for infinitely many $j$'s. Fix $\epsilon <1/2
$ and pick $\lambda _{0}<\lambda <1/2$ such that $\lambda ^{-(1-\epsilon )}<2
$. By our hypothesis and this observation, we can take $N$ large and an
appropriate $k $ such that all the ratios $%
s_{k+1}/s_{k},...,s_{k+N}/s_{k+N-1}$ are greater than $\lambda ,$ but $%
s_{k+N+1}/s_{k+N}\leq \lambda $. Using (\ref{Gapsize}) gives the estimate 
\begin{align*}
2^{N}& \geq \lambda ^{-N(1-\epsilon )}\geq \left( \frac{s_{k}}{s_{k+N}}%
\right) ^{1-\epsilon }\geq 2^{N\epsilon }\left( \frac{s_{k}}{s_{k+N}}\right)
^{1-2\epsilon } \\
& \geq c2^{N\epsilon }\left( \frac{s_{k}}{a_{2^{k+N}}}\right) ^{1-2\epsilon }
\end{align*}%
where $c>0$ depends only on $\lambda $ and $\epsilon $. At this point we
proceed exactly as in the previous case, but letting $b_{i}$ denote the
number of gaps of level $k+N+1$. We omit the details.
\end{proof}

\subsection{The attainable values for the Assouad dimension}

\label{structure}

In this subsection, our interest is in determining the possible values for
the $\dim _{A}E$ when $E\in \mathscr{C}_{a}$ . We start with the following
example that shows that in general this set need not be an interval. In
fact, in this example, only two values are attained.

\begin{example}
\label{theexample} Assume $\{g_k\}_{k=0}^\infty$ is a strictly decreasing,
summable sequence satisfying

\begin{enumerate}
\item $\sum_{j\ge k+1} 2^jg_j< g_k/2$ (in particular, $2^kg_k\to0$);

\item $(g_{k-1}/g_k)^{1/k}\ge 2^{k+4}$.
\end{enumerate}

Let $a$ be the decreasing sequence where each $g_j$ occurs $2^j$ times. Then
each subset in $\mathscr C_a$ has either Assouad dimension $0$ or $1$.
\end{example}

\begin{proof}
Fix $E\in\mathscr C_a$. For each $k\ge1$, let $m_k$ be the maximum number of
complementary intervals of length $g_k$ between two consecutive gaps of
length strictly greater than $g_k$; here we consider also as gaps the
unbounded intervals $(-\infty,0)$ and $(L, \infty)$.

\begin{claim}
Suppose $\dim_A E<1$. Then $\{m_k\}$ is bounded.
\end{claim}

\begin{proof}[Proof of the Claim.]
Let $\dim_A E=\beta<\alpha<1$ and choose $c,\rho$ such that 
\begin{equation*}
N_r(B(x, R)\cap E)\le c\left(\frac{R}{r}\right)^\alpha \ \ \ \ \forall x\in
E, \ 0< r < R \le \rho. 
\end{equation*}
Choose $k_0$ such that $\rho\ge 2^{k_0}g_{k_0}$ and let $k\ge k_0$.

Consider a subinterval bounded by two consecutive gaps of length strictly
greater than $g_k$, where $m_k$ occurs. Let $x$ be an endpoint of the
subinterval, i.e., $x$ is an endpoint of one of the bounding gaps of size
greater than $g_k$. Let $R$ be the distance between the bounding gaps and $%
r=g_k/2$. As there are $m_k$ gaps of length $g_k$ between the two bounding
intervals and no gaps of size $g_j$, $j < k$, by the first hypothesis 
\begin{align*}
m_kg_k\le R &\le m_kg_k+\sum_{i>k}2^ig_i \\
&\le (m_k+1)g_k. \ \ \ \ \ \ \ 
\end{align*}
Note that $N_r(B(x,R)\cap E)\ge m_k$ since at least one $r$-ball is needed
for each left endpoint of intervals of size $g_k$. Hence 
\begin{align*}
m_k\le N_r(B(x,R)\cap E) \le c\left(\frac{R}{r}\right)^\alpha&\le c\left(%
\frac{(m_k+1)g_k}{g_k}\right)^\alpha \le 2^\alpha c m_k^\alpha.
\end{align*}
Thus $m_k^{1-\alpha}\le 2^\alpha c$. As $\alpha<1$ this proves the sequence $%
\{m_k\}$ is bounded, as claimed.
\end{proof}

Next, we show that any complementary set $E \in \mathscr C_a$, with $\{m_k\}$
bounded, has $\dim_A E=0$. To prove this, we will show that for each $%
\alpha=1/N$, there is a constant $c_N$ such that 
\begin{equation}  \label{condassouad}
N_r(B(x,R)\cap E)\le c_N\left(\frac{R}{r}\right)^\alpha \ \ \forall \
0<r<R\le 1.
\end{equation}
The choice of $c_N$ will be clear from the proof.

\noindent \emph{Case 1:} $g_N/2\le r<R\le 1$. Since 
\begin{equation*}
\sum_{i>N}2^ig_i< \frac{g_N}{2}\le r,
\end{equation*}
a ball of radius $r$ will cover the intervals lying between two consecutive
gaps of length greater than $g_{N}$. Hence 
\begin{equation*}
N_r(B(x,R)\cap E)\le 2\:\#\{\text{gaps} \ g_j, j\le N\}<2^{N+2}.
\end{equation*}
As long as $c_N\ge2^{N+2}$ then (\ref{condassouad}) holds in this case.

\noindent \emph{Case 2: $g_{k+1}/2<r\le g_k/2$, $R\ge g_{k-1}/2$ for some $%
k\ge N$.} To bound $N_r(B(x,R)\cap E)$ it will be sufficient to take all the
endpoints of gaps of size $g_j$, $j\le k+1$ by hypothesis $(1)$. Therefore 
\begin{equation*}
N_r(B(x,R)\cap E)\le 2\cdot 2^{k+2}+2,
\end{equation*}
the last $2$ to count the two endpoints on $[0, L]$. Moreover, by hypothesis 
$(2)$ 
\begin{align*}
\left(\frac{R}{r}\right)^\alpha\ge \left(\frac{g_{k-1}/2}{g_k/2}%
\right)^\alpha\ge \left(\frac{g_{k-1}}{g_k}\right)^{1/k}\ge 2^{k+4},
\end{align*}
and that together with the previous estimate shows that $c_N\ge 1$ works.

\noindent \emph{Case 3:} $r=g_k/2$, $g_k/2< R < g_{k-1}/2$. Consider $B(x,R)$%
. This interval cannot completely contain any gap of size $g_j$, $j\le k-1$
(although either endpoint of $B(x,R)$ could be contained in such a gap).
Thus $B(x,R)$ intersects $n_k\le m_k$ gaps of size $g_k$ and contains no full gaps of
size $g_i$, $i<k$. Hence 
\begin{equation*}
N_r(B(x,R)\cap E)\le 2n_k+2\le 2c_0+2
\end{equation*}
where $c_0$ is a bound for $\{m_k\}$. As $R/r>1$, we just pick $c_N>2c_0+2$
to get 
\begin{equation*}
N_r(B(x,R)\cap E)\le c_N\left(\frac{R}{r}\right)^{1/N}.
\end{equation*}

\noindent \emph{Case 4: $g_{k+1}/2<r< g_k/2$, $g_k/2 < R< g_{k-1}/2$.}
Again, $B(x,R)$ does not contain any gaps of size $g_j$, $j\le k-1$. If it
also does not fully contain any gaps of size $g_k$, then it intersects at
most $m_{k+1}$ gaps of size $g_{k+1}$ and it follows that 
\begin{equation*}
N_r(B(x,R)\cap E)\le 2m_{k+1}+2\le 2 c_0+2
\end{equation*}
and we are fine. Otherwise, it intersects at most $m_k$ gaps of size $g_k$.
Between any two of these (or between such a gap and the end of the interval $%
B(x, R)$), there are at most $m_{k+1}$ gaps of size $g_{k+1}$, so 
\begin{equation*}
N_r(B(x,R)\cap E)\le 2 (m_k+1)m_{k+1}+2m_k+2
\end{equation*}
which is again uniformly bounded.

\noindent \emph{Case 5: $g_{k+1}/2<r< g_k/2$, $g_{k+1}/2 < R< g_{k}/2$.}
This is essentially the same as the easy part of Case 4 as $B(x, R)$ admits
no gaps of size $g_j$, $j\le k$.

Finally, we remark that the Cantor rearrangement has $m_k=1$ so its Assouad
dimension is $0$. Moreover, $\dim_A D_a=1$.
\end{proof}

Next we show that under the doubling hypothesis on the gap sequence the set
of attained values is an interval.

\begin{theorem}
\label{central doubling interval} Let $C\{r_{j}\}$ be a central Cantor set
whose gap sequence is doubling. Then for any $s\in \lbrack \dim
_{A}C\{r_{j}\},1]$ there is a complementary set $E$ of the gap sequence of $%
C\{r_{j}\}$ such that $\dim _{A}E=s$.
\end{theorem}

\begin{proof}
We assume $s<1$ for otherwise we take $E$ to be the decreasing rearrangement
of Assouad dimension 1. We will construct $E$ in such a way that $E=\mathcal{%
A}\cup \mathcal{B}$, with $\dim _{A}\mathcal{A}=s$ and $\dim _{A}\mathcal{B}%
=\dim _{A}C\{r_{j}\}$. Because of the finite stability of the Assouad
dimension (see \cite{luukkainen1998assouad}) it will follow that $\dim
_{A}E=s$.

The set $\mathcal{A}$ will be the union of countably many finite sets $A_{k}$%
. Each $A_{k}$ will consist of the $2^{k+1}$ endpoints of the $k$'th step
intervals in a rescaled approximation of a central Cantor set of fixed ratio 
$\gamma <1/2 ,$ where $\gamma ^{s}=1/2$. (This choice is made so that the
central Cantor set with ratio $\gamma $ has Assouad dimension $s)$. We will
position the sets so that $\max A_{k+1}=\min $ $A_{k}$, with $0$ being the
unique accumulation point of $\mathcal{A}$.

Let $a=\{a_{j}\}$ be the gap sequence of $C\{r_{j}\}$. We will let $%
b=\{b_{j}\}$ be the subsequence of $a$ consisting of the gaps that remain
after forming $\mathcal{A}$. We will show that $b$ is equivalent to $a$ and
then call upon Proposition \ref{equivalentarrangement} to construct a
suitable set $\mathcal{B}$. We leave the details to the end of the proof.

We begin by fixing notation that will be used throughout the proof. Let $%
g_{k}=r_{1}\cdot \cdot \cdot r_{k-1}(1-2r_{k})$ be the length of the gaps at
step $k$ in the construction of $C\{r_{j}\}$ and let $\alpha >1$ be a
doubling constant for $C\{r_{j}\}$, meaning, 
\begin{equation*}
\frac{1}{\alpha }g_{k+1}\leq g_{k}\leq \alpha g_{k+1}\text{ for all }k.
\end{equation*}

The sets $A_{k}$ will be constructed by an iterative process. To start the
process, we let $d_{1}=g_{5}$ and $n_{1}+1=\max \{i:g_{i}\geq d_{1}\}$. Then 
$n_{1}>3$ and 
\begin{equation*}
g_{n_{1}+2}<d_{1}\leq g_{n_{1}+1}.
\end{equation*}%
Having chosen $d_{j}$ and $n_{j}>j+2$ for $j=1, \ldots,k-1,$ we now let $%
d_{k}$ be the size of a gap at level $i>k+2,$ sufficiently small so that

\begin{equation}
2^{k}\leq \left( \frac{d_{k-1}}{d_{k}}\right) ^{s}
\label{condition0sequence}
\end{equation}%
and%
\begin{equation}
\frac{\alpha ^2}{1-2\gamma }\sum_{i=j+1}^{k}d_{i}<d_{j}\gamma ^{j}\text{ for 
}j=1, \dots,k-1.  \label{condition00}
\end{equation}

Put $n_{k}+1=\max \{i:g_{i}\geq d_{k}\}$. This choice ensures that $%
n_{k}>k+2 $, the sequence $\{n_{k}\}$ is increasing and 
\begin{equation*}
g_{n_{k}+2}<d_{k}\leq g_{n_{k}+1}\text{ for all }k\text{.}
\end{equation*}

\emph{Construction of $A_{k}$.} Having chosen $d_{k}$ and $n_{k}$ we now
proceed with the construction of $A_{k}$. We note that (\ref{condition00})
implies 
\begin{equation}
\frac{\alpha }{1-2\gamma }\sum_{i=1}^{\infty }d_{j+i}<d_{j}\gamma ^{j}\text{
for all }j  \label{condition1sequence}
\end{equation}%
and 
\begin{equation}
\frac{\alpha ^2 }{1-2\gamma }d_{k+1}<d_{k}\gamma ^{k}\text{.}
\label{condition2}
\end{equation}

\noindent \emph{Step $1$.} For each $0\leq j\leq k,$ we will choose $2^{j}$
gaps from $\{a_{i}\}$ of lengths comparable to $d_{k}\gamma ^{j}$ in length
and call these the \emph{gaps of the $j$-th level in }$A_{k}$. To make such
a selection, we choose positive integers $i_{j},$ for $j=0,...,k,$ so that 
\begin{equation*}
n_{k}+i_{j}=\max \{l:g_{l}\geq d_{k}\gamma ^{j}\}.
\end{equation*}%
This will ensure that $i_{0}=1$, $i_{j}\leq i_{j+1}$, 
\begin{equation}
g_{n_{k}+i_{j}+1}<d_{k}\gamma ^{j}\leq g_{n_{k}+i_{j}}\text{ and }%
d_{k}\gamma ^{j}\leq g_{n_{k}+i_{j}}\leq \alpha d_{k}\gamma ^{j}\text{,}
\label{gapdefn}
\end{equation}%
the latter by the doubling condition.

We define the gaps of the $j$'th level in $A_{k}$ to have length $%
g_{n_{k}+i_{j}}$. As $n_{k}>k+2$, there are sufficiently many gaps from
which to make the selection, even if the $i_j$ are not strictly increasing.
In fact, we will use at most half the available gaps at any level.

\noindent \emph{Step $2$.} We arrange these $2^{k+1}-1$ gaps to produce the
discrete set $A_{k}$ with $2^{k+1}$ points in the following manner. Consider
the full binary tree $\mathcal{T}_{k}$ up to level $k$. For each $j=0, \ldots,k$ place the $%
2^{j} $ indices of the gaps of level $j$ in the vertices of level $j$ of $%
\mathcal{T}_{k}$. Then paste the corresponding gaps from left to right using
the in-order tree traversal; see Figure \ref{FigureTree} to recall the
definition of this order when $k=2$. This will be the set $A_{k}$; see
Figure \ref{FigureA3}. Since $A_{k}$ is made with $2^{j}$ gaps of level $j$
for $0\leq j\leq k$, by (\ref{gapdefn}) we have 
\begin{equation}
d_{k}\leq g_{n_{k}+1}\leq \text{diam}\: A_{k}\leq
\sum_{j=0}^{k}2^{j}g_{n_{k}+i_{j}}\leq \alpha d_{k}\sum_{j=0}^{k}2^{j}\gamma
^{j}<\frac{\alpha }{1-2\gamma }d_{k}.  \label{diamAk}
\end{equation}

\begin{figure}[tbp]
\centering
\par
\begin{tikzpicture}[level distance=1.5cm,
  level 1/.style={sibling distance=3cm},
  level 2/.style={sibling distance=1.5cm}]
  \node {4}
    child {node {2}
      child {node {1}}
      child {node {3}}
    }
    child {node {6}
    child {node {5}}
      child {node {7}}
    };
\end{tikzpicture}
\caption{The order given by $\mathcal{T}_2$.}
\label{FigureTree}
\end{figure}
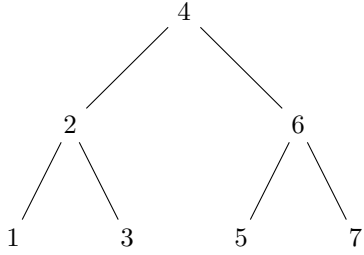

We remark that conditions (\ref{condition2}),(\ref{gapdefn}) and the
doubling property imply that 
\begin{equation*}
g_{n_{k}+i_{j}+2}\geq \frac{1}{\alpha }g_{n_{k}+i_{j}+1}\geq \frac{1}{\alpha
^2 }g_{n_{k}+i_{j}}\geq \frac{1}{\alpha ^2 }d_{k}\gamma ^{j}>d_{k+1},
\end{equation*}%
thus the gaps associated with sets $A_{k}$ and $A_{m}$, for $k\neq m$,
correspond to different gap levels in $C\{r_{j}\}$.

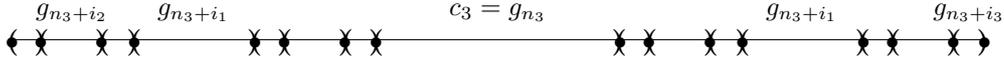
\begin{figure}[tbp]
\centering
\begin{tikzpicture}[scale=.4]
\draw (0,0.1) -- (32,0.1);
\draw(0.05,0) node{{\bf (}};
\draw(.95,0) node{{\bf )}};
\draw(1.05,0) node{{\bf (}};
\draw(2.95,0) node{{\bf )}};
\draw(3.05,0) node{{\bf (}};
\draw(4,0) node{{\bf )}};
\draw(4.1,0) node{{\bf (}};
\draw(7.95,0) node{{\bf )}};
\draw(8.05,0) node{{\bf (}};
\draw(8.95,0) node{{\bf )}};
\draw(9.05,0) node{{\bf (}};
\draw(10.95,0) node{{\bf )}};
\draw(11.05,0) node{{\bf (}};
\draw(11.95,0) node{{\bf )}};

\draw(0.025,0) node{$\bullet$};
\draw(.975,0) node{$\bullet$};
\draw(2.975,0) node{$\bullet$};
\draw(4.05,0) node{$\bullet$};

\draw(0.025+8,0) node{$\bullet$};
\draw(.975+8,0) node{$\bullet$};
\draw(2.975+8,0) node{$\bullet$};
\draw(4+8,0) node{$\bullet$};

\draw(12.05,0) node{{\bf (}};
\draw(19.95,0) node{{\bf )}};

\draw(16,1) node{$c_3=g_{n_3}$};
\draw(6,1) node{$g_{n_3+i_1}$};
\draw(26,1) node{$g_{n_3+i_1}$};
\draw(2,1) node{$g_{n_3+i_2}$};
\draw(31.5,1) node{$g_{n_3+i_3}$};

\draw(0.05+20,0) node{{\bf (}};
\draw(.95+20,0) node{{\bf )}};
\draw(1.05+20,0) node{{\bf (}};
\draw(2.95+20,0) node{{\bf )}};
\draw(3.05+20,0) node{{\bf (}};
\draw(4+20,0) node{{\bf )}};
\draw(4.1+20,0) node{{\bf (}};
\draw(7.95+20,0) node{{\bf )}};
\draw(8.05+20,0) node{{\bf (}};
\draw(8.95+20,0) node{{\bf )}};
\draw(9.05+20,0) node{{\bf (}};
\draw(10.95+20,0) node{{\bf )}};
\draw(11.05+20,0) node{{\bf (}};
\draw(11.95+20,0) node{{\bf )}};

\draw(0.025+20,0) node{$\bullet$};
\draw(.975+20,0) node{$\bullet$};
\draw(2.975+20,0) node{$\bullet$};
\draw(4.05+20,0) node{$\bullet$};

\draw(0.025+28,0) node{$\bullet$};
\draw(.975+28,0) node{$\bullet$};
\draw(2.975+28,0) node{$\bullet$};
\draw(4+28,0) node{$\bullet$};

\end{tikzpicture}
\caption{$A_3$ has $2^4$ points; the sizes of some gaps are indicated.}
\label{FigureA3}
\end{figure}

We set $\rho_{k}=d_{k}\gamma ^{k}$ and $R_{k}=\text{diam}\: A_{k}$. Both
sequences are decreasing, $\rho_{k}<R_{k}$ and by (\ref{condition1sequence})
and (\ref{diamAk}) we obtain the useful bounds 
\begin{equation}
R_{k+1}\leq \text{diam}\: \left( \bigcup_{j=k+1}^{\infty }A_{j}\right)
=\sum_{j=1}^{\infty }R_{k+j}\leq \frac{\alpha }{1-2\gamma }%
\sum_{j=1}^{\infty }d_{k+j}<d_{k}\gamma ^{k}=\rho_{k}<R_{k}.
\label{condition1sequence1}
\end{equation}

Having constructed the sets $A_{k}$, our next task is to prove that $\dim
_{A}\mathcal{A}=s$. In this part of the argument $c$ will denote a constant
that depends only upon $\alpha $ and $\gamma $.

By considering the special case $R=R_{k}$ and $r=\rho_{k}$, we first verify
that $s$ is a lower bound for the dimension. In fact, for any $y\in A_{k}$
the set $B(y,R)\cap A_{k}$ contains at least $2^{k}$ points spread $r$ apart
(one for each $k$'th level gap of $A_{k}$), hence the choice of $\gamma
=2^{-1/s}$ and (\ref{diamAk}) yields 
\begin{equation*}
N_{r}(B(y,R)\cap \mathcal{A})\geq 2^{k}=\gamma ^{-ks},
\end{equation*}%
while 
\begin{equation*}
\gamma ^{-ks}=\left( \frac{d_{k}}{d_{k}\gamma ^{k}}\right) ^{s}\leq \left( 
\frac{R_{k}}{\rho_{k}}\right) ^{s}\leq \left( \frac{\frac{\alpha }{1-2\gamma }%
d_{k}}{d_{k}\gamma ^{k}}\right) ^{s}=c\gamma ^{-ks}.
\end{equation*}%
Since the last quotient goes to infinity as $k\rightarrow \infty $, we have $%
\dim _{A}\mathcal{A}\geq s$.

To prove that $\dim _{A}\mathcal{A}\leq s$ we will show that there is a
constant $c$ such that for all $0<r<R\leq \text{diam}\: \mathcal{A}$ and $%
x\in \mathcal{A}$, $N_{r}(B(x,R)\cap \mathcal{A})\leq c(R/r)^{s}$.

Given such $r<R $, choose $m$ and $l\leq m+1$ such that 
\begin{equation*}
\rho_{m}\leq r<\rho_{m-1}\ \text{ and }\ R_{l}\leq R<R_{l-1}.
\end{equation*}%
From (\ref{condition1sequence1}) we deduce that $\bigcup_{j=m+1}^{\infty
}A_{j}$ is contained in a ball of radius $\rho_{m}\leq r,$ hence for all $x$
and $R$, 
\begin{equation}
N_{r}\left( B(x,R)\cap \bigcup_{j=m+1}^{\infty }A_{j}\right) \leq 1.
\label{ob1}
\end{equation}%
Another important observation is that all gaps in $A_{j}$ have length at
least $d_{j}\gamma ^{j}=\rho_{j}$, hence 
\begin{equation}
N_{r}\left( B(x,R)\cap \bigcup_{j=1}^{J}A_{j}\right) \leq 3 \ \ \text{ if } \ \ R\leq
\rho_{J}.  \label{ob2}
\end{equation}

We now divide the proof that $\dim _{A}\mathcal{A}\leq s$ into three cases.

\noindent \emph{Case 1:} $l<m$. Here we use (\ref{ob1}) and the observation
that for any set $F$, $N_{r}(F)\leq \#F$, to get the simple bound%
\begin{eqnarray*}
N_{r}(B(x,R)\cap \mathcal{A}) &\leq &N_{r}\left( B(x,R)\cap
\bigcup_{j=m+1}^{\infty }A_{j}\right) +N_{r}\left( B(x,R)\cap
\bigcup_{j=1}^{m}A_{j}\right) \\
&\leq &1+\sum_{j=1}^{m}\#A_{j}<8\cdot 2^{m-1}.
\end{eqnarray*}%
Furthermore, 
\begin{equation*}
\left( \frac{R}{r}\right) ^{s}\geq \left( \frac{R_{m-1}}{\rho_{m-1}}\right)
^{s}\geq \left( \frac{d_{m-1}}{d_{m-1}\gamma ^{m-1}}\right) ^{s}=2^{m-1},
\end{equation*}%
consequently, $N_{r}(B(x,R)\cap \mathcal{A})\leq 8(R/r)^{s}.$

\noindent \emph{Case 2a): $l=m$ and $R\leq \rho_{m-1}$.}

From (\ref{ob1}) and (\ref{ob2}) we deduce that $N_{r}\left( B(x,R)\cap
\bigcup_{j\neq m}A_{j}\right) \leq 4$, so it suffices to study $N_{r}\left(
B(x,R)\cap A_{m}\right) $.

If $r\geq d_{m}$, then $\text{diam}\: A_{m}=R_{m}\leq \alpha
d_{m}/(1-2\gamma )\leq cr$ and hence $N_{r}\left( B(x,R)\cap \mathcal{A}%
\right) \leq c+4$ and we are done.

Otherwise, as $r\geq \rho_{m}=d_{m}\gamma ^{m}$, there exists some $0\leq j<m$
such that 
\begin{equation*}
d_{m}\gamma ^{j+1}\leq r<d_{m}\gamma ^{j},
\end{equation*}%
and then 
\begin{equation*}
\left( \frac{R}{r}\right) ^{s}\geq \left( \frac{R_{m}}{d_{m}\gamma ^{j}}%
\right) ^{s}\geq \left( \frac{d_{m}}{d_{m}\gamma ^{j}}\right) ^{s}=2^{j}.
\end{equation*}%
The set $A_{m}$ is contained in the union of the disjoint closed intervals
bounded by consecutive gaps of $A_{m}$ of level $\leq j$. There are $2^{j+1}$
of these intervals and they each have length 
\begin{equation}
\sum_{l=1}^{m-j}2^{l-1}g_{n_{m}+i_{j+l}}\leq \sum_{l=1}^{m-j}2^{l-1}\alpha
d_{m}\gamma ^{j+l}<\frac{\alpha }{1-2\gamma }d_{m}\gamma ^{j+1}\leq cr.
\label{LengthGap}
\end{equation}%
Thus $N_{r}(B(x,R)\cap A_{m}) \leq c2^{j}$ and hence 
\begin{equation*}
N_{r}(B(x,R)\cap \mathcal{A})\leq c2^{j}\leq c\left( \frac{R}{r}\right) ^{s}.
\end{equation*}

\noindent \emph{Case 2 b): $l=m$ and $R>\rho_{m-1}.$}

As (\ref{condition1sequence1}) tells us $R_{m-1}<\rho_{m-2}$ and we are given $%
R<R_{m-1}$, from (\ref{ob2}) we deduce that 
\begin{equation*}
N_{r}\left( B(x,R)\cap \bigcup_{i\leq m-2}A_{i}\right) \leq 3.
\end{equation*}%
Thus it suffices to bound $N_{r}(B(x,R)\cap (A_{m-1}\cup A_{m}))$ by $%
c(R/r)^{s}$. Of course, $N_{r}(B(x,R)\cap (A_{m-1}\cup A_{m}))$ is dominated
by the cardinality of $A_{m-1}\cup A_{m},$ hence if $R\geq d_{m-1}$, then we
have 
\begin{equation*}
N_{r}(B(x,R)\cap \mathcal{A})\leq c2^{m}\leq c\left( \frac{d_{m-1}}{%
d_{m-1}\gamma ^{m-1}}\right) ^{s}\leq c\left( \frac{R}{r}\right) ^{s}.
\end{equation*}

So we can assume there is some $1\leq h<m$ such that $\ $%
\begin{equation*}
d_{m-1}\gamma ^{h}\leq R<d_{m-1}\gamma ^{h-1}.
\end{equation*}%
In this case, $R$ is smaller than the size of any gap in $A_{m-1}$ at level
less than $h$. It follows that $B(x,R)\cap A_{m-1}$ is contained in the
closed interval between two consecutive gaps of level $<h$ and therefore%
\begin{equation*}
\#\left( B(x,R)\cap A_{m-1}\right) \leq c2^{m-h}\text{.}
\end{equation*}%
If, also, $r\geq d_{m}$, then $\text{diam}\: A_{m}\leq cr$, so $%
N_{r}(A_{m})\leq c$. Putting these facts together we see that $%
N_{r}(B(x,R)\cap \mathcal{A})\leq C2^{m-h},$ while (\ref{condition0sequence}%
) ensures that%
\begin{equation*}
\left( \frac{R}{r}\right) ^{s}\geq \left( \frac{d_{m-1}\gamma ^{h}}{d_{m}}%
\right) ^{s}=2^{m-h}.
\end{equation*}

Otherwise, there exists $1\leq j<m$ such that 
\begin{equation*}
d_{m}\gamma ^{j+1}\leq r<d_{m}\gamma ^{j}.
\end{equation*}%
As in case 2 a), this hypothesis ensures $N_{r}(B(x,R)\cap A_{m})\leq c2^{j}$
so that 
\begin{equation*}
N_{r}(B(x,R)\cap \mathcal{A})\leq c(2^{m-h}+2^{j}),
\end{equation*}%
while $(R/r)^{s}\geq (d_{m-1}\gamma ^{h}/d_{m}\gamma ^{j})^{s}=2^{m+j-h}$.

\noindent \emph{Case 3: $l=m+1$.} Then we have $\rho_{m}\leq r<R<R_{m}<\rho_{m-1}$.

If $d_{m}\leq r,$ then $R\leq cr$, hence $N_{r}(B(x,R)\cap \mathcal{A})\leq c
$ and we are done.

So assume $d_{m}\gamma ^{j}\leq r<d_{m}\gamma ^{j-1}$ for some $1\leq j\leq
m $ and 
\begin{equation*}
d_{m}\leq R<R_{m}\ \ \ \text{or}\ \ \ d_{m}\gamma ^{h}\leq R<d_{m}\gamma
^{h-1}\text{ for some}\ 1\leq h\leq j.
\end{equation*}%
In either case, as $R<\rho_{m-1}$ inequalities (\ref{ob1}) and (\ref{ob2})
show that we need only study $N_{r}(B(x,R)\cap A_{m})$. The size of $R$
guarantees that $B(x,R)\cap A_{m}$ is contained in the interval between two
consecutive gaps of the $h$'th level in $A_{m}$ (or simply the interval
containing $A_{m}$ in the first case) and is contained in the union of the
closed subintervals bounded by gaps of level $\leq j$. From (\ref{LengthGap}%
), these subintervals have length $\leq cr$ and there are at most $2^{j+1-h}$
(with $h=0$ in the first case) of them. Hence 
\begin{equation*}
N_{r}(B(x,R)\cap A_{m})\leq c2^{j-h}=c\gamma ^{(h-j)s}\leq c\left( \frac{R}{r%
}\right) ^{s}.
\end{equation*}
This completes the proof that $\dim _{A}\mathcal{A}=s$.\medskip

To conclude the proof of the theorem we construct the set $\mathcal{B}$.
First, we establish that the sequences $a$ and $b$ are equivalent. Recall that $a_{2^k+i} = g_k$, for $0\le i < 2^k$, from the fact that the Cantor set is central. 
The choice of our construction implies that $b_j = a_{j+t_j}$ for some $0 \le t_j \le j$. Thus if $k$ is chosen with $2^k \le j < 2^{k+1}$, then either $2^k \le j+t_j < 2^{k+1}$, in which case $b_j = a_j = g_k$, or $2^{k+1} \le j + t_j < 2^{k+2}$, in which case $b_j = a_{j+t_j} = g_{k+1}$.
In either case, $1/\alpha \le b_j/a_j \le \alpha$ and so the sequences are equivalent. Now let $\mathcal{B}^{\prime }\in\mathscr C_b$ be the
complementary set corresponding to $C\{r_{j}\}\in \mathscr{C}_{a}$.
Proposition \ref{equivalentarrangement} implies that $\dim_A \mathcal{B}%
^{\prime }=\dim_A C\{r_{j}\}$. The set $\mathcal{B}$ will be a translate of $%
\mathcal{B}^{\prime }$ that lies to the right of $\mathcal{A}$. Then take $E=%
\mathcal{A}\cup\mathcal{B}\in\mathscr C_a$.
\end{proof}

\begin{remark}
We remark that this argument actually shows that given \textit{any} $s\in
(0,1)$ there is a set $\mathcal{A}$ whose complementary gaps are a subset of
those of $C\{r_{j}\}$ and which has Assouad dimension $s$.
\end{remark}

\begin{corollary}
Suppose $a$ is a decreasing, summable, doubling sequence. Then 
\begin{equation*}
\{\dim _{A}E:E\in \mathscr{C}_{a}\}=[\dim _{A}C_{a},1]
\end{equation*}%
%
%
%
%
%
%
%
\end{corollary}

\begin{proof}
As $a$ is doubling, the sequence of gaps of the central Cantor set $%
C\{r_{j}\}$ with ratios $r_{k}$ given by (\ref{centralCantor}) is equivalent
to $a$. In particular, $C\{r_j\}$ is doubling and $\dim _{A}C\{r_{j}\}=\dim
_{A}C_{a}$. Thus given any $s\in \lbrack \dim _{A}C_{a},1]$ the theorem
implies there is a rearrangement $E^{\prime }$ of the gaps of $C\{r_{j}\}$
that has Assouad dimension $s$. The corresponding rearrangement $E\in %
\mathscr{C}_{a}$ also has dimension $s$ by Proposition \ref%
{equivalentarrangement}.

The fact that $\dim_A C_a$ is a lower bound on the set of attainable values
of the Assouad dimension is the content of Theorem \ref{LowerBdDec}.
\end{proof}


\section{Lower Assouad dimensions of complementary sets}

In this section we will prove that $C_a$ has the maximal Lower Assouad
dimension among all the complementary sets in $\mathscr{C}_a$ and that under
a mild technical assumption, the set of attainable Lower Assouad dimensions
from the class $\mathscr{C}_a $ is the interval $[0,\dim_L C_a ]$.

\label{proofs lower dimension}

\begin{theorem}
\label{upper bound lower dim} If $a$ is any summable, decreasing sequence,
then $\dim _{L}E\leq \dim _{L}C_a$ for any complementary set $E\in %
\mathscr{C}_{a}$.
\end{theorem}

We begin with a technical lemma.

\begin{lemma}
Suppose $\alpha =\dim _{L}C_{a}\in (0,1)$. Then for any $0<\epsilon
<1-\alpha $ there exist $\lambda <1/2$ and arbitrarily large $k$ and $m$
such that $s_{k+1}/s_{k}\leq \lambda $ and 
\begin{equation*}
2^{m}<\left( \frac{s_{k}}{s_{k+m}}\right) ^{\alpha +\epsilon }.
\end{equation*}
\end{lemma}

\begin{proof}
Since $\alpha >0$, we have $t:=\inf s_{j+1}/s_{j}>0$ and from formula (\ref%
{formula lower}) we deduce that for arbitrarily large $M$ and $K$, 
\begin{equation*}
2^{M}<\left( \frac{s_{K}}{s_{K+M}}\right) ^{\alpha +\epsilon }. 
\end{equation*}%
Pick $\lambda <1/2$ and $b<1$ so that $\lambda ^{\alpha +\epsilon }b>1/2$.
If $s_{K+1}/s_{K}\leq \lambda $ we are done. So assume otherwise and let $%
H=\max \{1\leq i\leq M:s_{K+j}/s_{K+j-1}>\lambda $ for $j=1, \ldots, i\}$.
Then 
\begin{equation}
2^{M}<\left( \frac{s_{K}}{s_{K+M}}\right) ^{\alpha +\epsilon
}<2^{H}b^{H}\left( \frac{s_{K+H}}{s_{K+M}}\right) ^{\alpha +\epsilon }\leq 
\frac{2^{H}b^{H}}{t^{(M-H)(\alpha +\epsilon )}},  \label{thisothereq}
\end{equation}
so 
\begin{equation*}
\left( 2t^{\alpha +\epsilon }\right) ^{M-H}\leq b^{H}\text{.}
\end{equation*}%
This proves that not only is $H<M,$ but also that $\sup_{M}(M-H)=\infty $.
Letting $k=K+H$ and $m=M-H$, we obtain by (\ref{thisothereq}) that 
\begin{equation*}
2^{m}=2^{M-H}<b^{H}\left( \frac{s_{k}}{s_{k+m}}\right) ^{\alpha +\epsilon }.
\end{equation*}%
Since $s_{k+1}/s_{k}\leq \lambda $ by the definition of $H$ and $b<1$, the
lemma holds.
\end{proof}

\begin{proof}[Proof of Theorem \protect\ref{upper bound lower dim}]
We assume that $\alpha =\dim _{L}C_a<1$ for otherwise there is nothing to
prove. Let $\epsilon >0$ be such that $\alpha +\epsilon <1$. If $\alpha =0$
then, as noted in Corollary \ref{CoroDimLzero}, $\inf s_{j+1}/s_{j}=0$, so there
are arbitrarily large $k$ such that 
\begin{equation}
2<\left( \frac{s_{k}}{s_{k+1}}\right) ^{\epsilon }.  \label{zerodim}
\end{equation}%
Otherwise, $\alpha >0$ and by the previous lemma there is a constant $%
\lambda <1/2$ such that we can always find arbitrarily large $k$ and $m$
such that $s_{k+1}/s_{k}\leq \lambda $ and 
\begin{equation*}
2^{m}<\left( \frac{s_{k}}{s_{k+m}}\right) ^{\alpha +\epsilon }.
\end{equation*}

Put $R=2\delta s_{k}$ where $\delta =1/2-(1/2)^{1/\epsilon }$ in the first
case and $\delta =1/2-\lambda $ in the second; $R$ is at most the length of
any gap of level $\leq k-1$. If we let $r=s_{k+m}$ where $m=1$ in the first
case, then 
\begin{equation*}
\left( \frac{R}{r}\right) ^{\alpha +\epsilon }=\left( \frac{2\delta s_{k}}{%
s_{k+m}}\right) ^{\alpha +\epsilon }>c2^{m}
\end{equation*}%
for some constant $c>0$. Note that we can arrange for $R$ to be arbitrarily
small and $R/r$ to be arbitrarily large.

Consider the sets $B(x_{j},R)\cap E$ where $x_{j}$ are the left endpoints of
the gaps of level $k-1$ in $E$. There are $2^{k}-1$ such sets and the choice
of $R$ ensures they are disjoint. Each set $B(x_{j},R)\cap E$ is a subset of
the union of the closed intervals which lie between consecutive gaps of
levels $1, \ldots,k+m$. The total length of these intervals, summed over all $j$%
, is $2^{k+m}s_{k+m}$ and there are at most $2^{k+m+1}$ of them. Each of
these intervals, $I_{i},$ can be covered by $1+\left( \left\vert
I_{i}\right\vert /r\right) $ balls of radius $r$, hence%
\begin{eqnarray*}
\sum_{j=1}^{2^{k}-1}N_{r}(B(x_{j},R)\cap E) &\leq &\sum_{i}\left( 1+\frac{%
\left\vert I_{i}\right\vert }{r}\right) \\
&\leq &\frac{1}{r}\sum_{i}\left\vert I_{i}\right\vert +2^{k+m+1}\leq
c2^{k+m}.
\end{eqnarray*}%
Consequently, there is some index $j$ such that 
\begin{equation*}
N_{r}(B(x_{j},R)\cap E)\leq c2^{m} \leq c\left(\frac{R}{r}\right) ^{\alpha
+\epsilon }.
\end{equation*}%
It
follows that $\dim _{L}E\leq \alpha $.
\end{proof}


\begin{theorem}
Suppose $\sup r_{j}<1/2$. Then given any $\alpha <\dim _{L}C\{r_{j}\}$ there
is some complementary set $E$ of the gap sequence of $C\{r_{j}\}$ such that $%
\dim _{L}E=\alpha $.
\end{theorem}

\begin{proof}
If $\alpha =0$, we simply let $E$ be any countable rearrangement and then $%
\dim _{L}E=0$.

So assume $\alpha >0$. Then $\dim _{L}C\{r_{j}\}>0$, hence $\inf r_{j}>0$
and so there is some $\varepsilon >0$ such that $\varepsilon
<r_{j}<1/2-\varepsilon $ for all $j$.

Choose $d<1/2$ such that $\alpha =\log 2/\left\vert \log d\right\vert $. As $%
\alpha <\dim _{L}C\{r_{j}\}$, formula (\ref{LowerDimAlt}) shows that there
is some $N$ such that for all $n\geq N$ and for all $k$,%
\begin{equation*}
\frac{\log 2}{\left\vert \log d\right\vert }<\frac{n\log 2}{\left\vert \log
\left( r_{k+1}\cdot \cdot \cdot r_{n}\right) \right\vert }.
\end{equation*}%
Taking $k=0$, we see that $s_{n}=r_{1}\cdot \cdot \cdot r_{n}>d^{n}$.

For each $j\geq N$, choose $k_{j}$ such that 
\begin{equation*}
s_{k_{j}}\leq d^{j}<s_{k_{j}-1}\text{.}
\end{equation*}%
As $r_{k_{j}}>\varepsilon $, we have $s_{k_{j}-1}\geq s_{k_{j}}\geq
\varepsilon s_{k_{j}-1}$, so $d^{j}\sim s_{k_{j}}$. The sequence $\{k_{j}\}$
is increasing and the fact that $s_{n}>d^{n}$ implies $k_{j}>j$. In
particular, if $k_{j}=\cdot \cdot \cdot =k_{j+m}$, then $k_{j}>j+m$ and $m$
is bounded.

We will form a Cantor set $C_{b}$ associated to a subsequence $b$ of the gap
sequence $a$ of $C\{r_{j}\}$. Indeed, this subsequence will be chosen in
such a way that the lengths of the $2^{j-N}$ gaps at step $j-N+1$ in the
construction of $C_{b}$ will be the lengths of the gaps of step $k_{j}$ of
the construction of $C\{r_{j}\}$. ($C_{b}$ will also be a central Cantor
set, but of diameter less than $1$.) We note that the property that $k_{j}>j$
ensures that we will use at most half the available gaps from each step $%
k_{j}$.

The gaps of step $j$ in $C\{r_{i}\}$ have length comparable to $r_{1}\cdot
\cdot \cdot r_{j}$, thus the gaps of step $j-N+1$ in $C_{b}$ will be
comparable in size to $s_{k_{j}}$ and hence to $d^{j}$. Therefore 
\begin{eqnarray*}
s_{n}^{(b)} &=&\frac{1}{2^{n}}\sum_{k=2^{n}}^{\infty }b_{k}=\frac{1}{2^{n}}%
\sum_{i=n+1}^{\infty }2^{i-1}\cdot \text{ length of gaps of step }i \\
&\sim &\frac{1}{2^{n}}\sum_{i=n+1}^{\infty }2^{i}d^{i+N}\sim d^{n}.
\end{eqnarray*}%
Consequently, $s_{k+n}^{(b)}/s_{k}^{(b)}\sim d^{n}$ and an application of
formula (\ref{LowerDimAlt}) shows that $\dim _{L}C_{b}=\log 2/\left\vert
\log d\right\vert =\alpha $.

Now let $b^{\prime }$ be the sequence consisting of the remaining terms of
the gap sequence $a$ of $C\{r_{j}\}$. The same arguments as at the
conclusion of the proof of Theorem \ref{central doubling interval} show that 
$b^{\prime }$ is equivalent to $a$. Hence, the complementary set $A\in%
\mathscr C_{b^{\prime }}$ corresponding to $C\{r_{j}\}\in\mathscr C_a$ and $%
C\{r_{j}\}$ have the same Lower Assouad dimension. Define the set $E$ to be
the union of $C_{b}$ and a translation of $A$ that lies strictly to the
right of $C_{b}$. Since the Lower Assouad dimension of the union of two sets
with positive separation is the minimum of the Lower Assouad dimensions of
the two sets, $\dim _{L}E=\alpha $ (see \cite{Fraser}).
\end{proof}

\begin{rem}
The assumption that $\sup r_{j}<1/2$ is not a necessary condition for the
conclusion of the Theorem. It can be shown, for example, that if $%
r_{j}=1/2-1/(2j)$, then the same property holds for $C\{r_{j}\}$. We omit
the details.
\end{rem}

\begin{corollary}
Suppose $a$ is a decreasing, summable sequence and that $\sup
s_{j}/s_{j-1}<1/2 $. Then 
\begin{equation*}
\{\dim_L E : E \in \mathscr{C}_a\} = [0, \dim_L C_a].
\end{equation*}
\end{corollary}

\begin{proof}
Of course, $\dim_L D_a =0$ and the fact that $\dim_L C_a$ is an upper bound 
on the attainable dimensions is the content of Theorem \ref{upper bound
lower dim}.

So assume  $0 < \alpha < \dim_L C_a$. Then it must be the case that $\inf
s_{j}/s_{j-1}>0$ and it is a routine exercise to see that this fact,
combined with the assumption that $\sup s_{j}/s_{j-1}<1/2,$ implies $a$ is a
doubling sequence. Consequently, $a$ is equivalent to the gap sequence of a
central Cantor set with ratios $\{r_{j}\}$ as defined in (\ref{centralCantor}%
). Since $r_{1}\cdot \cdot \cdot r_{j}=s_{j}$ it is immediate that $%
C\{r_{j}\}$ satisfies the hypothesis of the proposition. The corollary
follows since Proposition \ref{equivalentarrangement} implies the Lower
Assouad dimensions of the sets in $\mathscr C_a$ have the same dimensions as
the corresponding complementary sets of the gap sequence of $C\{r_{j}\}$.
\end{proof}

%
%
%
%
%

We conclude with an application.

\begin{corollary}
Suppose there is some $E\in $ $\mathscr{C}_{a}$ such that $\dim _{L}E=\dim
_{A}E=s$. Then, also, $\dim _{L}C_{a}=\dim _{A}C_{a}=\dim _{H}C_{a}=s$.
\end{corollary}

\begin{proof}
From our theorems,%
\begin{equation*}
\dim _{L}E\leq \dim _{L}C_{a}\leq \dim _{H}C_{a}\leq \dim _{A}C_{a}\leq \dim
_{A}E,
\end{equation*}%
and thus the hypothesis implies we have equality throughout.
\end{proof}

\begin{rem}
In particular, if $E\in $ $\mathscr{C}_{a}$ is Ahlfors regular, then $\dim
_{L}E=\dim _{A}E$ and thus $\dim _{H}E=\dim _{H}C_{a}$, a fact which can
also be shown by other well known arguments.
\end{rem}

\begin{rem}
It would be interesting to understand the structure of the sets $\{\dim
_{A}E:E\in \mathscr{C}_{a}\}$ or $\{\dim _{L}E:E\in \mathscr{C}_{a}\}$ when $%
a$ is not a doubling sequence (or does not satisfy the separation condition $%
\sup s_{j}/s_{j-1}<1/2$ in the second case). It particular, we do not know
if it is possible to obtain the full interval $[0,1]$ in either case, or the
two element set $\{0,1\}$ in the second case.
\end{rem}


\end{document}